\documentclass[a4paper, 11pt]{amsart}
\pdfoutput=1 

\usepackage[T1]{fontenc}
\usepackage{lmodern}
\usepackage{microtype}
\usepackage[margin=2.5cm]{geometry}

\usepackage[style=numeric, citestyle=numeric-comp, backend=biber, sorting=nyt, maxbibnames=99,
            isbn=false]{biblatex}
\bibliography{refs}

\appto{\bibsetup}{\raggedright}  
\DeclareFieldFormat{doi}{%
  \mkbibacro{DOI}\addcolon\nobreakspace
  \ifhyperref
    {\href{https://doi.org/#1}{\nolinkurl{#1}}}
    {\nolinkurl{#1}}}

\usepackage{enumitem}

\usepackage{caption}
\usepackage{wrapfig}

\usepackage{url}
\usepackage[dvipsnames]{xcolor}

\definecolor{darkred}{RGB}{160,0,0}
\definecolor{darkpurple}{RGB}{120,0,120}
\definecolor{lightpurple}{RGB}{140,50,150}
\definecolor{darkblue}{RGB}{0,0,160}

\usepackage{tikz}
\usetikzlibrary{cd}

\usepackage{mathrsfs}

\usepackage[cachedir=minted]{minted}
\colorlet{CodeBackground}{black!5!white}
\setminted{encoding=utf8, mathescape, bgcolor=CodeBackground, fontsize=\footnotesize}
\setminted[macaulay2]{label={\TT{Macaulay2}}}

\usepackage{bm}
\usepackage{centernot}
\usepackage{dsfont}
\renewcommand{\mathbb}[1]{\mathds{#1}}

\usepackage{sansmath}

\newcommand{\NN}{\mathbb N}
\newcommand{\ZZ}{\mathbb Z}
\newcommand{\QQ}{\mathbb Q}

\newcommand{\PP}{\mathbb P}
\NewDocumentCommand{\FF}{e{_}}{\mathbb F\IfValueT{#1}{_{\mkern-3mu#1}}}
\newcommand{\card}[1]{\lvert#1\rvert}
\newcommand{\ol}[1]{\overline{#1}}

\newcommand{\supp}{\operatorname{supp}}
\newcommand{\tinv}{\mathsf{t}}
\newcommand{\Exp}{\operatorname{Exp}}
\newcommand{\mideal}{\mathfrak m}
\newcommand{\nideal}{\mathfrak n}

\usepackage{mathtools}
\newcommand{\defas}{\coloneqq}

\newcommand{\Set}[1]{\left\{\,#1\,\right\}}

\usepackage{amsmath, amssymb, amsthm}
\renewcommand{\to}{\longrightarrow}
\renewcommand{\emptyset}{\varnothing}

\usepackage{hyperref}
\hypersetup{
  colorlinks,
  citecolor=lightpurple,
  filecolor=black,
  linkcolor=darkblue,
  urlcolor=darkblue
}
\theoremstyle{definition}
\newtheorem{theorem}{Theorem}[section]
\newtheorem{lemma}[theorem]{Lemma}
\newtheorem{proposition}[theorem]{Proposition}

\newtheorem{remark}[theorem]{Remark}
\newtheorem{corollary}[theorem]{Corollary}

\newtheorem{definition}[theorem]{Definition}

\newcommand{\TodoColor}[2]{\colorlet{TodoColor#1}{#2}}
\TodoColor{Default}{darkred}
\TodoColor{Outline}{darkblue}
\TodoColor{Thoughts}{darkpurple}

\NewDocumentEnvironment{TodoList}{oO{TODO\IfNoValueF{#1}{~(#1)}}+b}{%
  \begingroup%
  \noindent\sffamily\color{TodoColor\IfNoValueTF{#1}{Default}{#1}}%
  \textbf{#2:}
  \ifthenelse{\isundefined{\notodo}}{%
    \begin{itemize}
    #3
    \end{itemize}%
  }{[...]}%
  \endgroup%
}{}

\usepackage{xspace}
\NewDocumentCommand{\todo}{oO{TODO\IfNoValueF{#1}{~(#1)}}+m}{%
  \begingroup%
  \noindent\sffamily\color{TodoColor\IfNoValueTF{#1}{Default}{#1}}%
  \textbf{#2:} #3%
  \endgroup%
  \xspace%
}

\title{Trinomial containment in polynomial ideals is undecidable}

\author{Tobias Boege}
\address[Tobias Boege]{
UiT The Arctic University of Norway, Tromsø, Norway}
\email{post@taboege.de}

\author{Anna Hofer}
\address[Anna Hofer]{Otto-von-Guericke-Universität, Magdeburg, Germany}
\email{anna.hofer@ovgu.de}

\author{Thomas Kahle}
\address[Thomas Kahle]{Otto-von-Guericke-Universität, Magdeburg, Germany}
\email{thomas.kahle@ovgu.de}

\date{\today}

\subjclass[2020]{Primary
  03D35; 
  secondary
  03D25, 
  11D09, 
  11U05, 
  13F20, 
  13P10} 
\keywords{fewnomial, sparse polynomial, ideal membership, Hilbert's tenth
problem, computably enumerable, Pell equation}

\begin{document}

\begin{abstract}
  We prove that deciding whether an ideal in a polynomial ring
  contains a trinomial is impossible on a Turing machine.  More
  precisely, from an integer polynomial $P$ we compute generators of
  an ideal $I_P$ in a polynomial ring over $\QQ$ such that $I_P$
  contains a trinomial if and only if $P$ has an integral zero. By the
  MRDP~theorem this problem is undecidable.
  A universal halting polynomial gives a computable family of ideals
  in one fixed polynomial ring, with uniform bounds on colength,
  generator count, and generator degree, for which the containment
  of a trinomial encodes the halting problem.
\end{abstract}

\maketitle

\section{Introduction} \label{sec:intro}

Every polynomial $f$ in the indeterminates $x = (x_1, \dots, x_m)$ has
a unique representation
\[
  f = \sum_{a \in \supp(f)} c_a x^a, \qquad c_a \neq 0,
\]
where the \emph{support} $\supp(f) \subseteq \NN^m$ is the set of exponent
vectors appearing with nonzero coefficient.  The summands $c_a x^a$
are the \emph{terms} of~$f$, and the \emph{number of terms} of $f$
is $|\supp(f)|$.
Any polynomial here is understood in this reduced form, that is, the same 
monomial never appears in two or more terms.
Let $K$ be a field.  For an ideal
$I \subseteq K[x_1, \dots, x_m]$ in the polynomial ring over~$K$, let
\[
  \tinv(I) \defas \min \Set{ \card{\supp(f)} : 0 \neq f \in I }, \qquad
  \tinv(0) \defas \infty.
\]
The quantity $\tinv(I)$ is the length of a shortest nonzero
polynomial in~$I$. The shortest polynomials are monomials, binomials
and trinomials (having exactly one, two and three terms, respectively).
For a fixed integer $k \in \NN$, we consider the computational decision
problem in which one receives a finite list of generators of $I$ as the
input and must decide whether $\tinv(I) \le k$ holds.

Properties of the support of polynomials appear much more
computationally delicate than the more common degree-based invariants
such as the Betti table.  Deciding $\tinv(I) \le 1$ is easy, since $I$
contains a monomial if and only if
$I : (x_1 \cdots x_m)^\infty = (1)$.  Already for two terms,
number-theoretic aspects become visible, but the solution is still
effective.  The \emph{binomial part} of~$I$, the vector space spanned
by all binomials in~$I$, is again an ideal and computable over number
fields~\cite{JensenKahleKatthan2017,KreuzerWalsh2023}, so
$\tinv(I) \leq 2$ is decidable.

The main result Theorem~\ref{thm:main} implies that the problem
becomes undecidable for trinomials. By an \emph{undecidable problem} we
throughout mean a Turing-undecidable problem, that is, a class of YES/NO
questions for which we seek a single Turing machine that correctly answers
all questions in the class (and in particular always terminates).
In his survey~\cite{Poonen2014}, Poonen argues that
undecidable problems arise naturally in mathematics.
Hilbert's~10th problem is a prominent example. It asks for an algorithm
that given a diophantine equation, correctly answers if the equation has
an integral solution. The MRDP~theorem, named after Davis, Putnam, Robinson,
and Matiyasevich~\cite{DavisPutnamRobinson1961,Matiyasevich1970},
shows that such an algorithm cannot exist.

Our problem class for this paper consists of all ideals
$I \subseteq \QQ[x_1, \dotsc, x_m]$ for arbitrary $m$.
Thanks to Hilbert's Basis Theorem, every such ideal is finitely generated
and can be encoded into a finite string. Our undecidability result follows
by a direct reduction from Hilbert's~10th problem.
For an integral polynomial $P \in \ZZ[y_1, \dots, y_n]$ write
$\mathcal Z_\ZZ(P) \defas \Set{ y \in \ZZ^n : P(y) = 0 }$ for its set
of integral zeros. The construction of Theorem~\ref{thm:main} computes
from~$P$ a finite list of generators of a proper zero-dimensional
ideal $I_P \subseteq \QQ[S, T, D_1, \dots, D_N]$ such that
\[
  \tinv(I_P) =
  \begin{cases}
    3, & \mathcal Z_\ZZ(P) \neq \emptyset, \\
    4, & \mathcal Z_\ZZ(P) = \emptyset.
  \end{cases}
\]
By the MRDP theorem it is undecidable if
$\mathcal Z_\ZZ(P) = \emptyset$ and hence deciding $\tinv(I_P) \leq 3$
is impossible too.  The theorem says slightly more.  Even separating the
two consecutive values $\tinv = 3$ and $\tinv = 4$ is impossible on a
Turing machine.  Moreover, $I_P$ contains neither a monomial nor a binomial,
so ``at most three'' and ``exactly three'' agree on the instances produced
by the reduction.
More undecidability consequences are collected
in~Corollary~\ref{cor:undecidable} and in Section~\ref{sec:universal}
we describe consequences of the existence of universal diophantine
equations in this context.

To delineate the complexity, for a \emph{fixed} finite set $E \subseteq \NN^m$
of exponent vectors it is decidable
whether $I$ contains a nonzero polynomial with support contained in~$E$.
This is what Gr\"obner bases do.  From~$I$ one computes the residues 
of the monomials $x^a$ modulo $I$ for all $a \in E$ and tests them for
linear dependence.
The finite exponent sets~$E$ can be enumerated effectively.
Enumerating rational coefficient vectors as well,
a program can even print all trinomials of~$I$ in some order,
but no finite amount of waiting can certify that $I$ contains no
trinomial.  Our results also show that no degree bound for the
shortest polynomial in an ideal~$I$ is computable from a
representation of~$I$, since a finite search up to such a bound
would decide the problem.

\subsection*{Historical context and related work}
\enlargethispage{3ex}
The number of terms is a curious complexity measure for polynomials.
It appears naturally when considering sparse representations for
polynomials, but it also has some counter-intuitive properties, for
example that because of cancellation, the number of terms of a
product is not predictable or at least depends on properties of the
coefficients, as is visible in the intricate theory of
decompositions of binomial ideals~\cite{KahleMiller2014}.

\subsubsection*{Sparse polynomials}

In sparse or lacunary computer algebra, a polynomial is represented by
its nonzero coefficients and their exponent vectors.
Thus, the number of terms is directly relevant for algorithmic
complexity. Plaisted used this representation for early hardness results
on polynomial and divisibility problems.
Computing the degree of the least common multiple of a finite set of
sparse univariate integer polynomials is NP-hard~\cite{Plaisted1977},
and so is deciding whether two sparse univariate polynomials have
a nontrivial common factor~\cite[Theorem~3.3]{Plaisted1984}.
Fewnomial theory uses support size as a degree-independent complexity measure.
Its aim is to bound the number of nondegenerate positive real solutions of
systems whose supports are prescribed~\cite{Khovanskii1980,BihanSottile2007}.
Both of these instances use sparsity as a complexity measure but they do
not deal with the question of finding short polynomial relations.

\subsubsection*{Unit equations and algebraic dynamics}

Finding a $k$-nomial in an ideal $I \subseteq K[x_1, \dots, x_m]$
is the same as finding a Laurent $k$-nomial in its extension~$IL$
inside the Laurent ring $L = K[x_1^\pm, \dots, x_m^\pm]$. From the
Gr\"obner basis point of view, one would like to fix the exponent
vectors of a putative $k$-nomial and use normal forms and linear
algebra to see if there are coefficients from $K$ which produce an
ideal element.
There is a dual point of view in which the coefficients are fixed
and fitting Laurent monomials are sought. This is the setting of
\emph{unit equations}:
\begin{equation}
  \label{eq:unit}
  a_1 u_1 + a_2 u_2 + \dots + a_k u_k = 0.
\end{equation}
This equation is interpreted in a $K$-algebra~$A$ with fixed
coefficients $a_1, \dots, a_k \in K$ and unknowns $u_1, \dots, u_k$
from a finitely generated multiplicative subgroup
$\Gamma \subseteq A^\times$.
For short polynomials in $IL$ one picks $A = L/IL$ and $\Gamma$
the image of the group of Laurent monomials.
The origins of this topic lie in algebraic number theory where $K$
is a number field, $S$ is a finite set of places of $K$ and $\Gamma$
is the unit group of the ring of $S$-integers in~$K$; see~\cite{EvertseGyory}.

The theory of unit equations furnishes finiteness results and in some
cases effective bounds on the exponent vectors solving~\eqref{eq:unit}.
Notably, if $K$ has positive characteristic and $I$ is prime, the
solution set admits an effective description~\cite{DM1}. Over finite fields, the coefficients $a_i$ can
be absorbed into the finitely generated group~$\Gamma$ and hence the
short polynomial question is governed by a single instance of~\eqref{eq:unit}
in which all~$a_i = 1$. A decision procedure for the short polynomial
problem for general ideals with finite field coefficients was
therefore anticipated in~\cite{DM3}.
A recent result of Dong and Shafrir~\cite{DongShafrir} allows a
transparent proof of this (Proposition~\ref{prop:finite-fields}).

Unit equations also matter in \emph{algebraic dynamics}.
In this area it was discovered that $\ZZ^d$-actions on compact abelian
groups are (in a categorical sense) equivalent to countable modules
over the Laurent polynomial ring $\ZZ[u_1^\pm, \dotsc, u_d^\pm]$.
See~\cite{AlgebraicOrigin} for a textbook-size introduction.
Derksen and Masser write of ``the problem of finding the
`shortest' polynomial in a given ideal'' in the context of mixing,
adding that ``in zero characteristic this problem is surprisingly
difficult and probably there is in general no effective
algorithm''~\cite[p.~2628]{DM3}.  Theorem~\ref{thm:main} confirms
their suspicion.

\subsubsection*{Sparse multiples}

The direct search for the shortest polynomial is interesting already
in the univariate situation, where every ideal is principal.  Let
$I = (f) \subseteq K[x]$.  Finding a short element means finding a
sparse polynomial multiple of~$f$.  Because any sparse multiple can be
multiplied by $x$ to produce another, one usually considers only
polynomials not divisible by $x$ (or works in the Laurent polynomial~ring).
Trinomials of this kind are named \emph{standard trinomials}.  They~have
the form $ax^m - bx^n - c$, with $m > n > 0$ and $a, b, c$ all
nonzero.  Standard trinomials that differ by a nonzero scalar factor
are identified, so that counting them is meaningful.  In early work
Posner and Rumsey asked which rational polynomials $p$ divide
infinitely many standard trinomials.  Making such $p$ is easy enough:
modulo a quadratic polynomial the residues of $1$ and~$x$ span the
entire ring, so any three powers of~$x$ are linearly dependent.
Therefore, for every $r \geq 1$, every divisor of a quadratic in
$x^r$ divides infinitely many trinomials, provided its constant term
is nonzero.  Posner and
Rumsey almost proved the converse, namely that every divisor of
infinitely many standard trinomials divides a cubic in $x^r$ for
some~$r$.  They conjectured that a linear or quadratic polynomial
in~$x^r$ always suffices~\cite{PosnerRumsey1965}.
Gy{\H{o}}ry and Schinzel proved this quantitatively.  If $p$ divides
more than $(4sd)^{s^6 2^{180d} + 8s\ell}$ standard trinomials, the
conclusion already follows.  Here $d$ is the degree of the splitting
field of~$p$, $\ell$ its number of distinct roots, and $s$ a count of
associated places.  The bound is independent of the coefficient
sizes~\cite[Theorem~1]{GyorySchinzel1994}.  Schlickewei and Viola
removed the dependence on $s$ and~$d$ entirely, replacing the bound by
$2^{44000} (\deg p)^{1000}$~\cite{SchlickeweiViola1997}.  Over
\emph{any} field of characteristic zero, $p$ divides infinitely many
standard trinomials if and only if it divides a linear or quadratic
polynomial in~$x^r$ for some~$r$~\cite[Theorem~2A]{GyorySchinzel1994}.
Positive characteristic is genuinely different.  An analogue for
dividing $k$-nomials fails for $k \geq 4$~\cite{GyorySchinzel1994}.

Giesbrecht, Roche, and Tilak worked on the algorithmic problem of
finding short multiples.  For an irreducible polynomial~$f$,
\cite[Theorem~3.3]{GiesbrechtRocheTilak2012} gives an explicit search
bound on the degree of a binomial multiple.  Assembling the factors of
an arbitrary rational polynomial,
\cite[Theorem~3.6]{GiesbrechtRocheTilak2012} then yields a
deterministic algorithm returning the least-degree binomial multiple,
or a proof that none exists. For $k$-nomial multiples with $k \geq 3$
over~$\QQ$, their algorithm needs an a priori height bound as input
and the additional hypothesis that $f$ has no repeated cyclotomic
factors.  They call removing these restrictions desirable ``though not
necessarily possible'', and suspect the problem is NP-complete
over~$\QQ$ once $k$ is part of the
input~\cite[Section~6]{GiesbrechtRocheTilak2012}. Bilu and Luca give
effective degree and height bounds under further conditions on the roots.
Our Proposition~\ref{prop:squarefree-univariate} records a resulting
decision procedure for squarefree univariate ideals over number
fields.

\subsubsection*{Short polynomials on algebraic sets}

Instead of having an arbitrary ideal $I$ given by generators, one may
consider a class of structured subsets in affine or projective spaces
and ask for the shortest polynomial vanishing there.
For general and skew-symmetric matrices of bounded rank,
there are no shorter polynomials in their vanishing
ideals than the appropriate determinants and Pfaffians,
respectively~\cite{DraismaKahleWiersig2023}.
The (numerical) search for short polynomials vanishing on algebraic
sets is part of interpolation theory.  See for
example~\cite{HauensteinMatusevichPetersonSherman2022}.
In a geometric setting, the ideals in question are at least radical.
Our undecidability result makes essential use of nonreducedness and
does not exclude decidability of the $k$-nomial containment for
radical ideals.

\subsection*{Outline of the reduction}
Exponent vectors of Laurent monomials
$D^d = D_1^{d_1} \cdots D_N^{d_N}$ are a convenient encoding of
integer vectors $d \in \ZZ^N$.  Our reduction consists of making ideal
membership of a trinomial whose exponents store $d$ equivalent to a
diophantine condition on~$d$.  Doing this requires three things:
sufficiently many candidate trinomials indexed by integer vectors, a
mechanism that imposes polynomial equations on those vectors using
only computable ideal operations, and a translation of Hilbert's 10th
problem into equations that the mechanism can impose.  These are the
contents of Sections~\ref{sec:base}, \ref{sec:quadrics},
and~\ref{sec:encoding}, respectively.

In Section~\ref{sec:base} we construct the \emph{base ideal} $J_0$ in
the Laurent ring
$L_N = \QQ[S^{\pm1}, T^{\pm1}, D_1^{\pm1}, \dots, D_N^{\pm1}]$.  By
Proposition~\ref{prop:shape} it contains no monomial and no binomial,
and up to multiplication by nonzero scalars and Laurent monomials its
trinomials form exactly two families: the \emph{affine} trinomials
$\tau_d = 1 - SD^d - TD^{-d}$, one for each $d \in \ZZ^N$, and the
\emph{infinite} trinomials $\theta_e^{p,q}$, indexed by a primitive
vector $e \in \ZZ^N$ and distinct nonzero integers $p, q$. The affine
family is the supply we want.  It realizes every integer vector as a
trinomial. The infinite family is unwanted, but cannot be avoided
(Remark~\ref{rem:infinite-forced}).  The ideal $J_0$ also contains the
quadrinomial $\Omega = (S+T-1)(S-T)$.

Section~\ref{sec:quadrics} provides the mechanism. To keep the
construction computable, all conditions are imposed
ideal-theoretically, by intersecting $J_0$ with kernels of ring
homomorphisms.  Any such intersection selects from the trinomials in
$J_{0}$, but the conditions we need concern the exponents of a
trinomial.  This problem is solved with an exponential map.

Section~\ref{sec:encoding} provides the translation. Hilbert's 10th
problem starts from a single polynomial $P \in \ZZ[y_1, \dots, y_n]$
of arbitrary degree which we write as a straight-line program by
making each intermediate step in the evaluation a new variable.  This
rewrites $P(y) = 0$ as a system of equations of degree at most two
with the same integral solvability (Lemma~\ref{lemma:quadratic}).
These equations must then be homogenized to become quadratic forms
required for the mechanisms in Section~\ref{sec:quadrics}.  The
homogenization in this step creates solutions at infinity which we
need to control with additional quadrics, so-called \emph{guards},
one of which arises from the classical Pell-equation.
In the end the guarded system has no nonzero rational solutions at
infinity, and an integral affine solution exactly when $P$ has an
integral zero.  Intersecting $J_0$ with the appropriate kernels,
one for each equation of the guarded system, kills
every infinite trinomial and retains the affine trinomial $\tau_d$
exactly when $d$ encodes an integral zero of~$P$.  If there is no
integral zero, no trinomial survives at all.  The quadrinomial
$\Omega$ lies in every one of the kernels and survives throughout,
which is what caps the answer at four.  Clearing of denominator type
arguments bring everything to the polynomial ring.  The desired ideal
$I_P$ is the kernel of one explicit
homomorphism to a finite-dimensional algebra, and linear algebra over
$\QQ$ computes a generating set of polynomials of total degree at
most five.  Corollary~\ref{cor:undecidable} collects
different formulations of the undecidability that follows from the
MRDP theorem.  In Section~\ref{sec:universal} we briefly explore the
consequences of the existence of universal diophantine equations.
Section~\ref{sec:effective} closes with finite procedures for
squarefree univariate ideals
and for arbitrary ideals over finite fields.

\subsection*{AI in this research}
Coding agents and large language models, the melange called ``AI'' in
2026, made substantial contributions to the results presented here.
The three authors
have discussed the possibility of a reduction from Hilbert's 10th
problem on multiple occasions. In the end it was GPT 5.6~Sol (OpenAI)
which came up with the first valid proof in July~2026.
The present paper is the result of a process in which the authors have
examined, strengthened, explained and rewritten the proof over many
iterations.  This has disentangled the argument and removed
redundancies as well as unnecessary dependencies.
Despite publishing papers, giving talks and getting funding granted
for the short polynomials problem, some important connections of this
problem to other areas of mathematics and computer science remained
hidden until AI pointed them out to us.  We have attempted to give
them appropriate coverage above.

Finally, a Lean formalization of our proof was almost entirely
generated by coding agents.  Next to GPT 5.6 Sol, also Claude Fable~5
(Anthropic) was used.  The authors have carefully vetted the
statements and their formalizations.  The formalization is available
on Microsoft GitHub at
\begin{center}
  \url{https://github.com/tom111/trinomials-undecidable-lean}.
\end{center}

\section{The base ideal} \label{sec:base}

We work in a Laurent polynomial ring
$L_N = \QQ[S^{\pm1}, T^{\pm1}, D_1^{\pm1}, \dots, D_N^{\pm1}]$ where
monomials are units.
Our construction takes place in the following explicitly given
\emph{base ideal} in~$L_N$:
\begin{equation}
  \label{eq:J0-gens}
  \begin{aligned}
    J_0 ={}
      &\bigl(S+T-1,\, (S-T)^5\bigr) + {} \\
      &\bigl((S-T)(D_i-1) : 1\leq i\leq N\bigr) + {} \\
      &\bigl((D_i-1)(D_j-1) : 1\leq i\leq j\leq N\bigr).
  \end{aligned}
\end{equation}

This ideal is best understood in the coordinates
\[
  A \defas S+T-1, \quad B \defas S-T, \quad \text{and} \quad C_i \defas D_i - 1.
\]
In these coordinates the Laurent ring $L_N$ is the localization of
$\QQ[A, B, C_1, \dots, C_N]$ at $S = (1+A+B)/2$, $T = (1+A-B)/2$, and
$D_i = 1+C_i$.  There
$J_0 = (A, B^5, B C_i, C_i C_j : 1 \leq i \leq j \leq N)$
is a monomial ideal which is primary to the maximal ideal
$(A, B, C_1, \dots, C_N)$.  Each inverted element has
nonzero constant term and therefore lies outside this maximal ideal,
so the localization preserves the quotient, the primaryness, and the
radical.  In $L_N$ the radical
\[
  \sqrt{J_0} = \left(2S-1,\, 2T-1,\, D_1-1,\, \ldots,\, D_N-1\right)
\]
is maximal. We record the following additional facts.

\begin{lemma} \label{lemma:A0-dim} The ideal $J_0$ has Krull
  dimension~zero.  A $\QQ$-basis of $A_0 \defas L_N / J_0$ consists of
  the images of the monomials
  $1, B, B^2, B^3, B^4, C_1, \dots, C_N$, and thus $\dim_{\QQ}(A_0) = N + 5$.
\end{lemma}

\begin{lemma} \label{lemma:J0-nf}
Modulo $J_0$ we have
\begin{equation*}
  S = \frac{1 + B}{2}, \quad
  T = \frac{1 - B}{2}, \quad \text{and} \quad
  D_i = 1 + C_i.
\end{equation*}
In particular, the images of $1+B$ and $1-B$ in $A_0$ are~units.
For any $a, b \in \ZZ$ and $d \in \ZZ^N$:
\[
  S^a T^b D^d \equiv 2^{-a-b}\left( (1+B)^a (1-B)^b + \sum_{i=1}^N d_i C_i \right) \pmod{J_0}.
\]
\end{lemma}

\begin{proof}
The expressions for $S$, $T$, and $D_i$ follow immediately from the
generators of~$J_0$. Then,
\begin{equation*}
  S^a T^b D^d \equiv 2^{-a-b} (1+B)^a (1-B)^b \prod_{i=1}^N (1+C_i)^{d_i} \pmod{J_0}.
\end{equation*}
Since each $1 + C_i$ is a unit modulo $J_0$ and $C_i C_j \equiv 0
\pmod{J_0}$, the binomial theorem for integer exponents gives
$(1+C_i)^{d_i} \equiv 1 + d_i C_i \pmod{J_0}$, so the product
$\prod_{i=1}^N (1+C_i)^{d_i}$ simplifies to $1 + \sum_{i=1}^N d_i C_i
\pmod{J_0}$. With $B C_i \equiv 0 \pmod{J_0}$ we get
\begin{align*}
  S^a T^b D^d &\equiv 2^{-a-b} (1+B)^a (1-B)^b \left(1 + \sum_{i=1}^N d_i C_i\right) \pmod{J_0} \\
  &\equiv 2^{-a-b} (1+B)^a (1-B)^b + 2^{-a-b} \sum_{i=1}^N d_i C_i \pmod{J_0}. \qedhere
\end{align*}
\end{proof}

\begin{definition}
We distinguish two special types of Laurent trinomials in $L_N$:
\begin{itemize}
  \item \emph{Affine type}: $\tau_d = 1 - S D^d - T D^{-d}$ with $d \in \ZZ^N$.
  \item \emph{Infinite type}: $\theta^{p,q}_e = (p-q) + q D^{pe} - p D^{qe}$ for
    a primitive vector $e \in \ZZ^N$ and distinct nonzero integers
    $p$ and~$q$.
\end{itemize}
For $d \in \ZZ^N$ write $|d| \defas (|d_1|, \dots, |d_N|) \in \NN^N$. The
\emph{clearing} of $\tau_d$ is the ordinary polynomial
\[
  \tilde\tau_d \defas D^{|d|}\,\tau_d = D^{|d|} - S D^{|d|+d} - T D^{|d|-d}
  \in \QQ[S, T, D_1, \dots, D_N].
\]
\end{definition}

\begin{proposition} \label{prop:shape}
The short polynomials in $J_0$ are as follows:
\begin{enumerate}[label=\textup{(\roman*)}]
\item \label{prop:shape-1-2} There is no monomial or binomial in the
  base ideal.
\item \label{prop:shape-3} Up to units of~$L_N$, its trinomials are
  precisely the affine and infinite trinomials.
\item \label{prop:shape-4} The quadrinomial $\Omega = (S+T-1) (S-T) =
  S^2 - T^2 - S + T$ lies in~$J_0$.
\end{enumerate}
\end{proposition}

\begin{proof}
  Claim~\ref{prop:shape-4} is obvious. For \ref{prop:shape-1-2}, there
  cannot be monomials in $J_{0}$ because these are units and $J_{0}$
  is proper. Up~to units, each binomial takes the form
  $f = 1 - c S^a T^b D^d$. The condition $f \in J_0$ is equivalent to
  the vanishing of each coefficient of the residue $f + J_0 \in A_0$
  in the basis in Lemma~\ref{lemma:A0-dim}.  Here and below we expand
  $(1\pm B)^a \equiv \sum_{k=0}^{4} \binom{a}{k} (\pm B)^k \pmod{J_0}$
  for any $a \in \ZZ$ by the binomial theorem, since $1 \pm B$ is a
  unit and $B^5 \equiv 0$ modulo $J_0$.  As always,
  $\binom{a}{k} = \frac{a(a-1)\cdots(a-k+1)}{k!}$ also for
  negative~$a$.
  Using Lemma~\ref{lemma:J0-nf} and comparing coefficients of the $1$,
  $C_i$, $B$, and $B^2$ gives
  \[
    c = 2^{a+b},\quad d = 0,\quad a - b = 0,\quad \text{and}\quad (a-b)^2 - a - b = 0.
  \]
  Thus $a = b = 0$ and $c = 1$, so $f = 0$ is not a true binomial.

  For claim~\ref{prop:shape-3} first use Lemma~\ref{lemma:J0-nf}
  and direct calculation to see that each $\tau_d$ and
  $\theta_e^{p,q}$ lies in~$J_0$.
  For the converse, multiply by suitable units to write a trinomial in
  $J_0$ as follows:
  \[
    f = 1 - c_1 S^{a_1} T^{b_1} D^{d_1} - c_2 S^{a_2} T^{b_2} D^{d_2}.
  \]
  For $i \in \{1,2\}$ let $\lambda_i = c_i 2^{-a_i-b_i}$. If $f$ is an
  honest trinomial, $\lambda_{1}$ and $\lambda_{2}$ are nonzero.
  Rewriting and comparing coefficients with Lemmas~\ref{lemma:A0-dim}
  and~\ref{lemma:J0-nf} yields
\begin{align}
  \label{eq:shape-3-1}
  \lambda_1 d_1 + \lambda_2 d_2 &= 0, \quad\text{and}\quad \\
  \label{eq:shape-3-2}
  \lambda_1 (1+B)^{a_1} (1-B)^{b_1} + \lambda_2 (1+B)^{a_2} (1-B)^{b_2} &\equiv 1 \pmod{B^5}.
\end{align}
The relation~\eqref{eq:shape-3-1} shows that $d_1$ and $d_2$ are linearly
dependent.
For $k \in \{0, \dots, 4\}$ the $B^k$-coefficient of $(1+B)^a (1-B)^b$ is
\[
  \sum_{\ell=0}^k (-1)^\ell \binom{a}{k-\ell}\binom{b}{\ell}.
\]
Then \eqref{eq:shape-3-2} imposes one constraint per degree.
Let $r_i = a_i - b_i$ and $s_i = a_i + b_i$:
\begin{align*}
  \tag{$k=0$} \label{eq:k=0} \lambda_1 + \lambda_2 &= 1, \\
  \tag{$k=1$} \label{eq:k=1} \lambda_1 r_1 &= -\lambda_2 r_2, \\
  \tag{$k=2$} \label{eq:k=2} \lambda_1 (r_1^2 - s_1) &= -\lambda_2 (r_2^2 - s_2), \\
  \tag{$k=3$} \label{eq:k=3} \lambda_1 r_1 (r_1^2 - 3s_1 + 2) &= -\lambda_2 r_2 (r_2^2 - 3s_2 + 2), \\
  \tag{$k=4$} \label{eq:k=4} \lambda_1 (r_1^4 - 6 r_1^2 s_1 + 8 r_1^2 + 3 s_1^2 - 6 s_1) &= -\lambda_2 (r_2^4 - 6 r_2^2 s_2 + 8 r_2^2 + 3 s_2^2 - 6 s_2).
\end{align*}
This system of equations can be solved with computer algebra such as \texttt{Macaulay2}~\cite{M2}:
\begin{minted}{macaulay2}
R = QQ[l1,l2,a1,a2,b1,b2, r1,r2,s1,s2];
-- Express the B^i-coefficient system in a1,b1,a2,b2.
coeff = (k,a,b) -> sum toList apply(0 .. k, l -> (-1)^l * binomial(a, k-l) * binomial(b, l));
rel = k -> (rhs := if k == 0 then 1 else 0; coeff(k,a1,b1)*l1 + coeff(k,a2,b2)*l2 - rhs);
I = saturate(ideal toList apply(0 .. 4, rel), l1*l2);
-- Coordinate change to r1,s1,r2,s2.
J = eliminate(I + ideal(-r1 + a1-b1, -r2 + a2-b2, -s1 + a1+b1, -s2 + a2+b2), {a1,a2,b1,b2});
(P1, P2) = toSequence decompose J
\end{minted}
The equations generate an ideal and the conditions $\lambda_1, \lambda_2 \neq 0$
require a saturation. The resulting ideal has two associated primes which
give rise to two cases.

\textbf{Case 1}: The first prime $P_1$ has simply $r_1 = s_1 = r_2 = s_2 = 0$
and hence $a_1 = b_1 = a_2 = b_2 = 0$. Thus $f$ takes the form
\[
  1 - c_1 D^{d_1} - c_2 D^{d_2}.
\]
If $f$ is an honest trinomial then $c_1, c_2 \neq 0$ and $d_1$ and $d_2$
are nonzero, distinct and collinear by~\eqref{eq:shape-3-1}. Since
every nonzero integer vector is an integer multiple of a primitive
one, we may write $d_1 = p e$ with $p \in \ZZ$ nonzero and
$e \in \ZZ^N$ primitive. Then $d_2 = q e$ for $q = \frac{-\lambda_1 p}{\lambda_2}
= \frac{-c_1 p}{c_2} \in \QQ$. Since $e$ is primitive, Bézout's lemma
supplies $u \in \ZZ^N$ such that $\langle u, e\rangle = 1$. It then
follows that $q = q\langle u, e\rangle = \langle u, d_2\rangle \in \ZZ$.
From \eqref{eq:shape-3-1} and \eqref{eq:k=0} we have $c_1 p + c_2 q = 0$
and $c_1 + c_2 = 1$. Together they imply $c_1 = \frac{-q}{p-q}$ and
$c_2 = \frac{p}{p-q}$. Thus,
\[
  f = 1 + \frac{q}{p-q} D^{pe} - \frac{p}{p-q} D^{qe} = \frac{1}{p-q} \theta_e^{p,q}
\]
is of infinite type (up to a unit).

\textbf{Case 2}: The other prime $P_2$ has more complicated generators.
Using \mintinline{macaulay2}{eliminate(P2, {l1,l2})} we isolate relations
which do not involve $\lambda_1, \lambda_2$. The resulting elimination
ideal is generated by six polynomials:
\begin{align*}
  P_2 \cap \QQ[r_1,s_1,r_2,s_2] &{}={} \bigl(
    s_1^2 - s_1 s_2 + s_2^2 - 1,
    r_2^2 + 2 s_1 - s_2 - 2,
    r_1 s_2 + r_2 s_1 - r_2 s_2 + r_2, \\
  &\hphantom{{}={} \bigl(\,}
    r_1 s_1 - r_1 - r_2 s_2 + r_2,
    r_1 r_2 + s_1 + s_2 - 1,
    r_1^2 - s_1 + 2 s_2 - 2
  \bigr).
\end{align*}
This ideal contains
$s_1^2 - s_1 s_2 + s_2^2 - 1 = \frac14 ((2s_1 - s_2)^2 + 3s_2^2 - 4)$
and
$r_1^2 - r_1 r_2 + r_2^2 - 3 = \frac14 ((2r_1 - r_2)^2 + 3r_2^2 -
12)$.  The sum-of-squares parts confine the possible integer solutions
to $-1 \le s_1, s_2 \le 1$ and $-2 \le r_1, r_2 \le 2$.  Explicit
enumeration gives the complete list:
\begin{center}
\begin{tabular}{c||c|c||c|c||c|c}
$(a_1-b_1,a_1+b_1) = (r_1, s_1)$ & $(-2 ,  0)$ & $(-1 , -1)$ & $(-1 ,  1)$ & $( 1 ,  1)$ & $( 1 , -1)$ & $( 2 ,  0)$ \\ \hline
$(a_2-b_2,a_2+b_2) = (r_2, s_2)$ & $(-1 , -1)$ & $(-2 ,  0)$ & $(1  ,  1)$ & $(-1 ,  1)$ & $( 2 ,  0)$ & $(1  , -1)$
\end{tabular}
\end{center}
The solutions come in three pairs which are exchanged by swapping the
indices $1$ and $2$, that is, the order of two terms
of~$f$.  We may therefore assume that $((r_1, s_1), (r_2, s_2))$ is
one of
\[
  ((1,1), (-1,1)), \qquad ((2,0), (1,-1)), \qquad ((-2,0), (-1,-1)),
\]
so that $((a_1, b_1), (a_2, b_2))$ equals
\[
  ((1,0), (0,1)), \qquad ((1,-1), (0,-1)), \qquad ((-1,1), (-1,0)),
\]
respectively.  In the first case
$f = 1 - c_1 S D^{d_1} - c_2 T D^{d_2}$.  Equations \eqref{eq:k=0} and
\eqref{eq:k=1} give $\lambda_1 = \lambda_2 = \frac12$, thus
$c_1 = c_2 = 1$, and then $d_2 = -d_1$ by~\eqref{eq:shape-3-1}.  Hence
$f = \tau_{d_1}$ is of affine type.

In the second case
$f = 1 - c_1 S T^{-1} D^{d_1} - c_2 T^{-1} D^{d_2}$.  Multiplying by
the unit $-c_2^{-1} T D^{-d_2}$ gives
\[
  f' = 1 + \tfrac{c_1}{c_2}\, S D^{d_1 - d_2}
         - \tfrac{1}{c_2}\, T D^{-d_2},
\]
an honest trinomial in $J_0$ of the form treated in the first case.
Hence $f' = \tau_{d_1 - d_2}$ and $f$ is of affine type up to a unit.
The third case is symmetric to the second under exchanging $S$
and~$T$.  Multiplying
$f = 1 - c_1 S^{-1} T D^{d_1} - c_2 S^{-1} D^{d_2}$ by the unit
$-c_2^{-1} S D^{-d_2}$ again produces an honest trinomial of the form
treated in the first case.
\end{proof}

\section{Imposing quadratic constraints} \label{sec:quadrics}

We develop the main tool of the proof.  By means of intersection of
ideals we keep from $J_{0}$ only those trinomials that satisfy
quadratic equations on the exponents $d$ of $\tau_d$ and $e$ of
$\theta_e^{p,q}$, namely
\begin{align}
  \label{eq:Q-aff} Q_0(d) + \ell_0(d) + b_0 &= 0, \\
  \label{eq:Q-inf} Q_0(e) &= 0,
\end{align}
where $Q_0\colon \QQ^N \to \QQ$ is a quadratic form,
$\ell_0\colon \QQ^N \to \QQ$ a linear form, and $b_0 \in \QQ$. These
two equations arise from a single homogenized quadratic form $Q$
constructed as follows. Let $V = \QQ^{N+1} = \QQ^N \times \QQ v_0$,
where $v_0$ is a homogenization coordinate.  Denote the standard basis
vectors of $V$ by $v_1, \dots, v_N, v_0$.  We tacitly identify an
exponent vector $d \in \QQ^N$ with $\sum_{i=1}^N d_i v_i \in V$, so
that $d + t v_0 \in V$ has coordinates $(d, t)$.  The~vector $(d, 1)$
corresponds to the point $[d:1]$ of the projective space $\PP^N_\QQ$,
while for $e\neq 0$, the exponent $(e, 0)$ is the point $[e:0]$ on the
hyperplane at~infinity.
Let $Q\colon V \to \QQ$ be the quadratic form
\[
  Q(d, t) = Q_0(d) + \ell_0(d) t + b_0 t^2.
\]
Then $\text{\eqref{eq:Q-aff}} \iff Q(d, 1) = 0$ and
$\text{\eqref{eq:Q-inf}} \iff Q(e, 0) = 0$. These equations are imposed on
the trinomial exponents by intersecting $J_0$ with the kernel of an explicit
homomorphism $\phi_Q\colon L_N \to A_Q$. We~first construct its codomain
and then the map.

\begin{definition}
For a symmetric bilinear form $B\colon V \times V \to \QQ$ let $A_B = \QQ
\times V \times \QQ$. Equip this set with coordinatewise addition and a
multiplication given by
\[
  (r,v,s) \cdot (r',v',s') \defas \bigl(rr',\; rv' + r'v,\; rs' + r's + B(v,v')\bigr).
\]
Given a quadratic form $Q\colon V \to \QQ$ let $B(v,v') = \frac12 \left(
Q(v+v') - Q(v) - Q(v')\right)$ be its associated bilinear form and set
$A_Q \defas A_B$.
\end{definition}

The following facts are well-known and we omit the proof.
\begin{lemma} \label{lemma:A_B} Identify $V$ with its factor
  $\{0\} \times V \times \{0\} \subseteq A_B$ and let
  $\zeta = (0,0,1) \in A_{B}$.  The $\QQ$-algebra
  $A_B = \QQ \oplus V \oplus \QQ \zeta$ is graded and local with
  nilradical $\nideal = V \oplus \QQ\zeta$, and $\nideal^3 = 0$.  Let
  $v_1, \dots, v_n$ be a basis of $V$ and put $b_{ij} = B(v_i, v_j)$.
  Then $A_B \cong \QQ[X_1, \dots, X_n, Z] / I$ with
\[
  I = (X_i X_j - b_{ij}Z : 1\leq i\leq j\leq n) +
      (X_i Z : 1\leq i\leq n) +
      (Z^2).
\]
\end{lemma}

\begin{remark}
  The construction of $A_B$ is standard.  It is a local ring whose
  maximal ideal $\nideal$ satisfies $\nideal^3 = 0$.  Such rings
  appear, for example, as \emph{radical cube
    zero}~\cite{KikumasaYoshimura2003} or \emph{short local} rings.
  In commutative algebra, modules over such rings are a test case for
  understanding infinite free
  resolutions~\cite{AvramovIyengarSega2008}.  Their multiplication is
  a symmetric bilinear map
  $\nideal/\nideal^2 \times \nideal/\nideal^2 \to \nideal^2$.
  Conversely, every symmetric bilinear map of vector spaces arises
  this way.  The
  socle of $A_B$ is $\operatorname{rad}(B) \oplus \QQ\zeta$, so $A_B$
  is Gorenstein if and only if $B$ is nondegenerate.  In that case it
  is the apolar algebra of a smooth
  quadric~\cite[Chapters~1--2]{IarrobinoKanev1999}, a Poincar\'e
  duality algebra classified by its quadratic form~\cite{Sah1974}, and
  the tangent cone in Sally's study of Gorenstein
  singularities~\cite{Sally1980}.  Closest to our use is the work of
  Fels and Kaup on affinely homogeneous surfaces~\cite{FelsKaup2012},
  which pairs this algebra with the truncated exponential employed
  below.
\end{remark}

On the nilradical of $A_B$ we have the (truncated) exponential map
$\Exp(x) = 1 + x + \frac12 x^2$ which is an isomorphism between the
nilradical $\nideal$ as an additive group and $1 + \nideal$ as a
multiplicative group
(see~\cite[Proposition~8.1]{LenstraSilverberg2018}).  It satisfies the
following elementary properties for $v, v' \in V$:
\begin{gather*}
  \Exp(v) = 1 + v + \frac12 Q(v)\zeta, \qquad
  \Exp(-v) = \Exp(v)^{-1}, \qquad
  \Exp(v+v') = \Exp(v) \Exp(v').
\end{gather*}

We now return to the concrete setting of \eqref{eq:Q-aff}
and~\eqref{eq:Q-inf} where $V = \QQ^N \times \QQ v_0$ and $Q$ is a
quadratic form on~$V$.  The vectors $v_1, \dots, v_N, v_0$ denote the
standard basis of~$V$, so $\dim V = N+1$.  The presentation of
$A_Q$ from Lemma~\ref{lemma:A_B} then uses the $N+1$ variables
$X_1, \dots, X_N, X_0$, one for each basis vector, and also~$Z$.

\begin{definition}
Let $\phi_Q\colon L_N \to A_Q$ be the ring homomorphism given by
\begin{equation*}
  S \mapsto \frac12 \Exp(v_0), \quad
  T \mapsto \frac12 \Exp(-v_0), \quad
  D_i \mapsto \Exp(v_i).
\end{equation*}
\end{definition}
This is a well-defined ring homomorphism from a Laurent ring since
$\Exp(x)$ is a unit for $x\in V$.

\begin{proposition} \label{prop:A_Q-eval} For every $d \in \ZZ^N$,
  primitive $e \in \ZZ^N$, and distinct nonzero $p, q \in \ZZ$ we have
\begin{equation}
  \label{eq:A_Q-eval}
  \phi_Q(\tau_d) = -\frac12 \, Q(d, 1)\zeta \quad\text{and}\quad
  \phi_Q(\theta_e^{p,q}) = \frac12 \, pq(p-q)\, Q(e, 0)\zeta.
\end{equation}
In particular $\tau_d \in \ker \phi_Q \iff Q(d, 1) = 0$
and $\theta_e^{p,q} \in \ker \phi_Q \iff Q(e, 0) = 0$.
\end{proposition}

\begin{proof}
The two formulas in \eqref{eq:A_Q-eval} follow by direct calculation,
starting from the identity
\[
  \phi_Q(S^a T^b D^d) = 2^{-a-b} \Exp(d + (a-b)v_0).
\]
Thus, for the affine trinomials,
\begin{align*}
  \phi_Q(\tau_d) &= 1 - \frac12 \Exp(d + v_0) - \frac12 \Exp(-d - v_0) \\
  &= 1 - \frac12 \left(1 + d + v_0 + \frac12 Q(d, 1)\zeta\right) -
         \frac12 \left(1 - d - v_0 + \frac12 Q(-d, -1)\zeta\right) \\
  &= -\frac14\left(Q(d, 1) + Q(-d, -1)\right)\zeta = -\frac12 Q(d, 1)\zeta.
\end{align*}
For the infinite type trinomials we have
\begin{align*}
  \phi_Q(\theta_e^{p,q}) &= (p-q) + q\Exp(pe) - p\Exp(qe) \\
  &= (p-q) + \left(q + pqe + \frac{p^2 q}{2} Q(e, 0)\zeta\right) -
             \left(p + pqe + \frac{p q^2}{2} Q(e, 0)\zeta\right) \\
  &= \frac{p^2q}{2} Q(e, 0) \zeta - \frac{pq^2}{2} Q(e, 0)\zeta = \frac12 \, pq(p-q) \, Q(e, 0)\zeta. \qedhere
\end{align*}
\end{proof}

\section{Encoding diophantine equations} \label{sec:encoding}

The MRDP theorem shows that Hilbert's 10th problem is undecidable: there is
no Turing machine which receives an integer $n$ and a polynomial $P \in
\ZZ[y_1, \dots, y_n]$ as input and decides $\mathcal Z_\ZZ(P) \neq \emptyset$.
Our main theorem encodes this problem directly into the question of whether
an ideal contains a trinomial.
To do so, we first perform a preprocessing step.
The evaluation of $P(y_1, \dots, y_n)$ can be encoded as a sequence
of additions and multiplications by introducing auxiliary variables.
For instance, $P(y_1, y_2) = 3y_1^2 + y_1 y_2 + 1$ can be encoded as
\begin{gather*}
  x_1 = y_1 \cdot y_1, \quad
  x_2 = 3 \cdot x_1, \quad
  x_3 = y_1 \cdot y_2, \quad
  x_4 = x_2 + x_3, \quad
  x_5 = x_4 + 1.
\end{gather*}
Such an encoding is known as a \emph{straight line
  program}~\cite[Chapter~4]{AlgComplexity}.  Eventually we reach an
auxiliary variable $x_r$ which is forced to equal $P(y)$ and to
enforce the existence of an integer solution to $P(y)=0$ we add the
equation $x_r = 0$.

\begin{lemma} \label{lemma:quadratic}
There is an algorithm which computes to any $P \in \ZZ[y_1, \dots, y_n]$
integers $r, m$ and a system of equations
$f_1, \dots, f_m \in \ZZ[x_1, \dots, x_r]$ such that
each $f_i$ has degree at most two and
\[
  \exists y \in \ZZ^n : P(y) = 0 \iff \exists x \in \ZZ^r : f_1(x) =
  \dots = f_m(x) = 0.
\]
\end{lemma}

This replaces a single equation with a system of equations, each of degree
at most two, preserving integral solvability in both directions. To impose
them on the exponent vectors of trinomials in the base ideal~$J_0$ via
Proposition~\ref{prop:A_Q-eval}, these equations must be homogenized.
This step introduces spurious solutions at infinity and at first breaks
the equivalence of integral solvability, a defect the following lemma~remedies.

\begin{lemma} \label{lemma:Pell-guard}
Let $f_1, \dots, f_m \in \ZZ[x_1, \dots, x_r]$ be polynomials of degree
at most two and let $X \subseteq \ZZ^r$ denote their common zero set.
Adjoin six variables $h, k, u_1, u_2, u_3, u_4$ as well as another
homogenization variable~$t$.  To each $f_i$ associate its degree-two
homogenization $g_i(x,t) \defas t^2 f_i(x_1/t, \dots, x_r/t)$. Also add
the two quadratic equations
\begin{align*}
  g_{m+1}(h,k,t) &= h^2 - 3k^2 - t^2, \\
  g_{m+2}(x,h,u_1,u_2,u_3,u_4,t) &= \sum_{i=1}^r x_i^2 + u_1^2 + u_2^2 + u_3^2 + u_4^2 - ht.
\end{align*}
Let $\hat X \subseteq \ZZ^{r+7}$ denote the solution set of the resulting
homogeneous quadratic system $g_1, \dots, g_{m+2}$, whose points we write
as $(d, t)$ with $d = (x, h, k, u_1, \dots, u_4) \in \ZZ^{r+6}$. Then:
\begin{itemize}[itemsep=0.4em]
\item The origin is the only rational solution of
  $g_1, \dots, g_{m+2}$ with $t = 0$.
\item The projection map $\pi\colon \hat X_1 = \hat X \cap \{ t=1 \} \to X$ is
  surjective.
\item There is a computable section $s\colon X \to \hat X_1$ of $\pi$.
\end{itemize}
In particular, $X = \emptyset \iff \hat X_1 = \emptyset$.
\end{lemma}

\begin{proof}
A rational solution of the system with $t = 0$ satisfies
\[
  0 = g_{m+1}(h,k,0) = h^2 - 3k^2,
\]
and since $3$ is not a rational square, this forces $h = k = 0$.  And
$g_{m+2} = 0$ becomes
\[
  x_1^2 + \dots + x_r^2 + u_1^2 + u_2^2 + u_3^2 + u_4^2 = 0,
\]
so all remaining coordinates vanish as well and the point is the
origin.

Now let $x \in X$ and fix the affine chart $t=1$. The~goal is to
construct $h, k, u_1, u_2, u_3, u_4$ which complete $x$ to a point in
$\hat X_1$.  The~equation $0 = g_{m+1}(h,k,1) = h^2 - 3k^2 - 1$ is a
\emph{Pell equation}.  The classic book by
Barbeau~\cite[Section~3]{Barbeau} covers elementary properties and
constructions.  We~only need the lemma that $h^2 - 3k^2 = 1$ has
integral solutions with arbitrarily large~$h$. Choose such a solution
where $\Delta = h - \sum_{i=1}^r x_i^2 \geq 0$. The difference
$\Delta$ is a nonnegative integer and by Lagrange's four-squares
theorem it is the sum of four squares:
$\Delta = u_1^2 + u_2^2 + u_3^2 + u_4^2$. Now the vector
$\ol{x} = (x, h, k, u_1, u_2, u_3, u_4, 1)$ satisfies
$g_1, \dots, g_m$ as well as $g_{m+1}$ and $g_{m+2}$ by
construction. Thus it belongs to $\hat X_1$.
The choice can be made effective.  Take for $(h,k)$ the first term of the
Pell solution sequence generated by $2 + \sqrt 3$ with $h$ large enough,
and find the four-squares witness $(u_1,u_2,u_3,u_4)$ by exhaustive
search. Hence, there is a computable section~$s$.
\end{proof}

The two additional equations act as so-called \emph{guards} in the
sense of computer science.  Their sole purpose is to block unwanted
states, here nonzero rational solutions at infinity.  They achieve
this because their restrictions to the hyperplane $\{t = 0\}$, namely
$h^2 - 3k^2$ and a sum of squares, are \emph{anisotropic} quadratic
forms, i.e.~forms without nontrivial rational zeros.  The following
remark explains why guarding the hyperplane at infinity is essential
for the encoding.

\begin{remark} \label{rem:poonen}
The passage from affine to projective in Lemma~\ref{lemma:Pell-guard} has a
well-known counterpart over~$\QQ$, where no guards are necessary.  A
homogeneous system has a nontrivial rational zero if and only if it has a
nontrivial integral one~\cite[Section~7.4]{Matiyasevich1993}, and by an
effective argument of R.~Robinson, recorded
in~\cite[Remark~1.2(c)]{Poonen2009}, rational solvability of arbitrary systems
reduces to nontrivial rational solvability of homogeneous ones.  Poonen used
this to show that deciding the existence of rational points on smooth
projective varieties over~$\QQ$ is as hard as Hilbert's 10th problem
over~$\QQ$, whose decidability is a famous open
problem~\cite{PoonenNotices,KoymansPagano2026}.  This contrast explains the role of the guards.  The
exponent vector of an infinite trinomial $\theta_e^{p,q}$ is meaningful only
up to scaling, so the infinite trinomials of an ideal see the
\emph{rational} points of the quadrics on the hyperplane at infinity.
Without guards, rational points at infinity yield  infinite
trinomials in the constructed ideals independently of whether $P$ has an
integral zero.  Controlling them would amount to an instance of
Hilbert's 10th problem over~$\QQ$.  The guards of
Lemma~\ref{lemma:Pell-guard} empty the hyperplane at infinity of
rational points, and only the original integral problem remains.
\end{remark}

\begin{remark} \label{rem:infinite-forced}
To get undecidability of trinomial detection in ideals from Hilbert's 10th
problem, it would be easier to only encode affine equations~\eqref{eq:Q-aff}
and get rid of the infinite trinomials.
However, our choice of affine trinomials forces the infinite
trinomials.  Suppose that an ideal of $L_N$ contains affine trinomials
$\tau_{d-e}, \tau_{d}, \tau_{d+e}$ for some nonzero $e \in \ZZ^N$
(which it reasonably might if they are all solutions to some equation)
then it would necessarily also contain an infinite trinomial. We have
\[
  \tau_{d-e} + \tau_{d+e} - (D^e + D^{-e})\tau_d = 2 - D^e - D^{-e}.
\]
Then writing $e = g e'$ with $e'$ primitive and $g \geq 1$, we have
$2 - D^e - D^{-e} = \frac{1}{g} \theta_{e'}^{g,-g}$.
\end{remark}

\begin{theorem} \label{thm:main} There is an algorithm which computes
  to any polynomial $P \in \ZZ[y_1, \dots, y_n]$ an integer $N$ and a
  list of at most $\binom{N+7}{5}$ generators, each of total degree at
  most five, of an ideal $I_P \subseteq \QQ[S, T, D_1, \dots, D_N]$
  such that:
\begin{itemize}[itemsep=0.4em]
\item $I_P$ is primary to the maximal ideal
  $\mideal = \left(2S-1,\ 2T-1,\ D_1-1,\ \dots,\ D_N-1\right)$ and therefore proper and of Krull dimension zero,
\item $I_P$ contains no monomial and no binomial,
\item $I_P$ contains the quadrinomial $S^2 - T^2 - S + T$,
\item every trinomial in $I_P$ is a term times a cleared affine
  trinomial, and
\[
  \exists d\in \ZZ^N : \tilde\tau_d \in I_P \iff \exists y\in \ZZ^n : P(y) = 0.
\]
\end{itemize}
\end{theorem}

\begin{proof}
  Use Lemmas~\ref{lemma:quadratic}~and~\ref{lemma:Pell-guard} to
  transform $P$ into a system of homogeneous quadratic forms
  $Q_1, \dots, Q_M$ in the $N+1$ variables $(x_1, \dots, x_N, t)$,
  with homogenization variable~$t$, such that
  \[
    \exists d \in \ZZ^N : Q_1(d,1) = \dots = Q_M(d,1) = 0
    \iff \exists y \in \ZZ^n : P(y) = 0,
  \]
  and such that no nonzero rational $e \in \QQ^N$ satisfies
  $Q_1(e,0) = \dots = Q_M(e,0) = 0$.  The number $N$ of affine
  variables determines the Laurent ring
  $L_N = \QQ[S^{\pm1}, T^{\pm1}, D_1^{\pm1}, \dots, D_N^{\pm1}]$.

  For each quadratic form $Q_i$ compute the finitely presented algebra
  $A_{Q_i}$ from Section~\ref{sec:quadrics} and the homomorphism
  $\phi_{Q_i}\colon L_N \to A_{Q_i}$.  Put
  \[
    J_P \defas J_0 \cap \bigcap_{i=1}^M \ker \phi_{Q_i}
    \qquad\text{and}\qquad
    I_P \defas J_P \cap \QQ[S, T, D_1, \dots, D_N].
  \]
  Since $J_0 = \ker(L_N \to A_0)$, the contraction $I_P$ is the kernel
  of the unital homomorphism
  \[
    \psi\colon \QQ[S, T, D_1, \dots, D_N] \to
    A_0 \times \prod_{i=1}^M A_{Q_i}
  \]
  which sends each variable to the tuple of its images. Let
  $\mideal = (2S-1, 2T-1, D_1-1, \dots, D_N-1)$ be the
  maximal ideal from the statement. We claim
  $\mideal^5 \subseteq I_P$.  First, under $\QQ[S,T,D] \to A_0$ the
  generators of $\mideal$ map, respectively, to $B$, $-B$,
  $C_1, \dots, C_N$ by Lemma~\ref{lemma:J0-nf}, and thus
  $(B, C_1, \dots, C_N)^5 = 0$ in $A_0$ since
  $B^5 = B C_i = C_i C_j = 0$.  Under $\phi_{Q_i}$ they map to
  $\Exp(v_0)-1$, $\Exp(-v_0)-1$, $\Exp(v_j)-1$, which lie in the
  nilradical $\nideal$ of $A_{Q_i}$, and $\nideal^3 = 0$
  (Lemma~\ref{lemma:A_B}).  So every product of five generators of
  $\mideal$ maps to zero in every factor, i.e.\
  $\mideal^5 \subseteq \ker\psi = I_P$.

  Let $W \subseteq \QQ[S, T, D_1, \dots, D_N]$ be the span of all
  products of at most four generators of~$\mideal$.  The generators
  $2S-1$, $2T-1$, $D_1-1, \dots, D_N-1$ are affine coordinates on
  $\QQ[S, T, D_1, \dots, D_N]$, and the change of coordinates
  preserves total degree.  Hence the $\binom{N+6}{4}$ monomials of
  degree at most four in these coordinates form a basis of~$W$, the
  $\binom{N+6}{5}$ monomials of degree five generate $\mideal^5$, and
  $\QQ[S, T, D_1, \dots, D_N] = W \oplus \mideal^5$.  Since
  $\mideal^5 \subseteq I_P$, we have
  \begin{equation}
    \label{eq:m5pplusKernel}
    I_P = \mideal^5 + \ker(\psi|_W).
  \end{equation}
  The matrix of $\psi|_W$ with respect to the monomial basis of~$W$
  is computable.  The codomain has dimension $(N+5) + M(N+3)$ by
  Lemma~\ref{lemma:A0-dim} and $\dim_\QQ A_{Q_i} = N+3$, and the
  entries follow from Lemma~\ref{lemma:J0-nf} and from
  $\phi_Q(S^a T^b D^d) = 2^{-a-b} \Exp(d + (a-b) v_0)$ as in
  Proposition~\ref{prop:A_Q-eval}.  A basis of $\ker(\psi|_W)$ is
  computed by linear algebra over~$\QQ$, so by~\eqref{eq:m5pplusKernel}
  $I_{P}$ is generated by at most
  $\binom{N+6}{5} + \binom{N+6}{4} = \binom{N+7}{5}$ polynomials of
  total degree at most five.  Since $\psi(1) = 1 \neq 0$, the ideal
  $I_P$ is proper, and $\mideal^5 \subseteq I_P$ then gives
  $\sqrt{I_P} = \mideal$.  An ideal whose radical is maximal is
  primary, so $I_P$ is $\mideal$-primary and of Krull dimension zero.

  By Propositions~\ref{prop:shape}~and~\ref{prop:A_Q-eval}, up to
  units of~$L_N$, the trinomials in $J_P$ are precisely those affine
  and infinite type trinomials $\tau_d$ and $\theta_e^{p,q}$ which
  satisfy, respectively, $Q_i(d, 1) = 0$ and $Q_i(e, 0) = 0$ for all
  $i = 1, \dots, M$.  Since $J_P \subseteq J_0$,
  Proposition~\ref{prop:shape} also shows that $J_P$ contains no
  monomial and no binomial.
  Lemma~\ref{lemma:Pell-guard} ensures that no infinite trinomial
  survives and that some $\tau_d$ with $d \in \ZZ^N$ lies in $J_P$ if
  and only if $P$ has an integral zero.
  Each kernel $\ker\phi_{Q_i}$ contains the quadrinomial
  $\Omega = (S-T)(S+T-1) = S^2 - T^2 - S + T$.  Indeed,
  $\phi_{Q_i}(S+T-1) = \frac12 Q_i(v_0)\zeta$ and
  $\phi_{Q_i}(S-T) = v_0$, whose product vanishes since $\zeta V =
  0$.  As $\Omega \in J_0$ by Proposition~\ref{prop:shape} and
  $\Omega$ is an ordinary polynomial, $\Omega \in \ker\psi = I_P$.

  Since $D^{|d|}$ is a unit of $L_N$,
  $\tau_d \in J_P \iff \tilde\tau_d \in J_P$, and since
  $\tilde\tau_d$ is an ordinary polynomial,
  $\tilde\tau_d \in J_P \iff \tilde\tau_d \in I_P$.  Together with
  the criterion for $\tau_d \in J_P$ established above, this yields
  the equivalence claimed in the theorem.  Some cleared affine
  trinomial $\tilde\tau_d$ lies in $I_P$ if and only if $P$ has an
  integral zero.  To see that every trinomial of $I_P$ is a term
  times a cleared affine trinomial, let $f \in I_P$ be a trinomial.
  As a trinomial of $J_P$ it takes the form
  $f = c\, S^aT^bD^m \tau_d$ with a nonzero scalar~$c$, and since all
  three exponent vectors of the polynomial $f$ are nonnegative,
  $a, b \geq 0$ and $m \geq |d|$ componentwise.  Hence
  $f = c\, S^aT^bD^{m-|d|}\, \tilde\tau_d$ is a term times a cleared
  affine trinomial.  Since $I_P \subseteq J_P$ and $J_P$ contains no
  monomial and no binomial, neither does $I_P$.
\end{proof}

We now record further special properties of our construction. Most of
them are not direct corollaries to the statement of
Theorem~\ref{thm:main} but to its proof.  Two decision problems are
\emph{many-one equivalent} if each can be translated into the other by
a computable map of instances that preserves YES and NO answers but is
not necessarily one-to-one.

\begin{corollary} \label{cor:undecidable}
The following decision problems about ideals $I \subseteq \QQ[x_1, \dots, x_k]$,
presented by finite lists of generators, are many-one equivalent to
the halting problem:
\begin{enumerate}[label=(\Alph*), itemsep=0.4em]
\item \label{cor:Undec-3}      Whether $I$ contains a trinomial.
\item \label{cor:Undec-le-3}   Whether $I$ contains a nonzero polynomial with at most three terms.
\item \label{cor:Undec-3-or-4} Under the promise that $\tinv(I) \in \{3, 4\}$,
  whether $\tinv(I) = 3$.
\end{enumerate}
This equivalence persists if the ideals are restricted to be
zero-dimensional.
\end{corollary}

\begin{proof}
  These problems are computably enumerable and hence reduce to the
  halting problem.  The finite exponent sets can be enumerated
  effectively, and Gr\"obner bases and linear algebra on the residues
  decides whether $I$ contains a nonzero polynomial with support
  contained in, or equal to, a given set.

  For the converse reduction we start from Hilbert's 10th problem,
  which is many-one equivalent to the halting problem by the MRDP
  theorem. The algorithm of Theorem~\ref{thm:main} computes from $P$ a
  zero-dimensional ideal $I_P$ which contains the quadrinomial
  $\Omega$, no monomial, and no binomial, and which contains a
  trinomial if and only if $P$ has an integral zero. Hence all three
  questions, asked of $I_P$, have answer YES if and only if $P$ has an
  integral zero.
\end{proof}

\begin{remark}\label{rem:geometry}
Our construction preserves more than the bare solvability of the
diophantine equation. Write $\mathscr{T}(J_P)$ for the set of
trinomials in $J_P$ up to multiplication by units of~$L_N$. In~the
notation of Lemma~\ref{lemma:Pell-guard}, the proof of
Theorem~\ref{thm:main} yields a chain
\[
  \mathscr{T}(J_P) \xrightarrow{\ \sim\ } \hat X_1
  \overset{\pi}{\twoheadrightarrow} X \xrightarrow{\ \sim\ } Y
  = \Set{ y \in \ZZ^n : P(y) = 0 }.
\]
The bijection on the left sends the class of $\tau_d$ to $(d,1)$. The
bijection on the right is projection onto the $y$-coordinates: in the
system produced by Lemma~\ref{lemma:quadratic} each auxiliary variable
is a polynomial in the preceding ones, so the system defines a scheme
isomorphic over $\ZZ$ to the hypersurface $P = 0$. The surjection
$\pi$ of Lemma~\ref{lemma:Pell-guard} forgets the guard coordinates.
It admits a computable section, and each of its fibers is infinite,
parametrized by the sufficiently large solutions of the Pell equation
together with four-squares witnesses. Consequently $J_P$ contains
either no trinomial at all or infinitely many pairwise nonassociated
ones, one computably singled out for each integral zero of~$P$.
\end{remark}

\section{A universal halting family} \label{sec:universal}

Fix a G\"odel numbering $(M_e)_{e \in \NN}$ of Turing machines and let
\[
  K_0 \defas \Set{ e \in \NN : \text{$M_e(0)$ halts} }
\]
be the halting set, which is computably enumerable but not
decidable~\cite[Section~1.9]{Rogers1967}. The MRDP theorem, followed by the replacement of
natural number variables by sums of four integer squares, gives a
polynomial $U \in \ZZ[E, Y_1, \dots, Y_r]$ such that
\begin{equation} \label{eq:universal-polynomial}
  e \in K_0 \iff \exists y \in \ZZ^r : U(e, y) = 0.
\end{equation}

Putting the universal diophantine equation into our framework shows
that there is a single polynomial ring in which we can compute a family
of ideals whose trinomial containment encodes the halting problem.
The ideals are even zero-dimensional and primary with uniformly bounded
colength, generator count, and generator degree.
\begin{corollary} \label{cor:halting-family}
There exist a positive integer $N$ and a computable sequence of ideals
$I_e \subseteq \QQ[S, T, D_1, \dots, D_N]$, indexed by $e \in \NN$,
such that
\[
  \tinv(I_e) =
  \begin{cases}
    3, & \text{$M_e(0)$ halts}, \\
    4, & \text{$M_e(0)$ does not halt}.
  \end{cases}
\]
Each $I_e$ is primary to the maximal ideal
$\mideal = \left(2S-1,\ 2T-1,\ D_1-1,\ \dots,\ D_N-1\right)$, in
particular proper and zero-dimensional, has colength
$\dim_\QQ(\QQ[S, T, D_1, \dots, D_N] / I_e) \leq \binom{N+6}{4}$, and
is generated by at most $\binom{N+7}{5}$ polynomials of total degree
at most five.
\end{corollary}

\begin{proof}
Apply the construction of Theorem~\ref{thm:main} to the polynomials
$P_e(Y) \defas U(e, Y) \in \ZZ[Y_1, \dots, Y_r]$ and write
$I_e \defas I_{P_e}$.  Substituting the constant $e$ for the variable
$E$ in one fixed straight line program for $U$, without simplifying,
gives programs of the same shape for every~$e$.
Lemmas~\ref{lemma:quadratic}~and~\ref{lemma:Pell-guard} therefore
produce the same number of homogeneous quadratic forms in the same
$N+1$ coordinates, and only their coefficients depend on~$e$.  The
ring $\QQ[S, T, D_1, \dots, D_N]$ and the maximal ideal $\mideal$ are
thus the same for all~$e$, and the generators of $I_e$ are computed
from~$e$ by the algorithm of Theorem~\ref{thm:main}.

Theorem~\ref{thm:main} shows that $I_e$ is $\mideal$-primary and
generated by at most $\binom{N+7}{5}$ polynomials of total degree at
most five.  Its proof shows $\mideal^5 \subseteq I_e$, so the quotient
$\QQ[S, T, D_1, \dots, D_N]/I_e$ is spanned by the residues of the
$\binom{N+6}{4}$ monomials of degree at most four in the affine
coordinates $2S-1$, $2T-1$, $D_1-1, \dots, D_N-1$.  Finally, $I_e$
contains no monomial and no binomial, contains the quadrinomial
$\Omega$, and contains a trinomial if and only if $P_e$ has an
integral zero, which by~\eqref{eq:universal-polynomial} is the case if
and only if $M_e(0)$ halts.
\end{proof}

Four observations come for free with the universal family.

\textbf{Coefficients.}
Only the rational coefficients of the generators depend on $e$ in an
essential way, and their heights are necessarily unbounded: with
bounded heights there would be only finitely many presentations of
this format, and a finite lookup table would decide~$K_0$.

\textbf{Search bounds.}  No computable function of~$e$ bounds the
exponents of the smallest trinomial in~$I_e$.  Multiplying by a
monomial shifts the support of a trinomial and preserves membership
in~$I_e$, because the variables map to units in the codomain.  Only
the shape of the support matters, and it can be normalized so that
its coordinatewise minimum is zero.  If a total computable function
$B \colon \NN \to \NN$ were such that, whenever $I_e$ contains a
trinomial, it contains one with normalized support in
$[0,B(e)]^{N+2}$, then the finite search through the box would
decide~$K_0$.  Absence of a trinomial has no finite certificate
either.  A trinomial in $I_e$ certifies its own presence, verified by
linear algebra, but a decidable relation $V \subseteq \NN^2$ with
$\tinv(I_e) = 4 \iff \exists c \in \NN : V(e,c)$ would, by trying
all $c$ in order, enumerate the complement of~$K_0$, which is not
computably enumerable.  A certificate scheme for arbitrary ideals
restricts to the family, so none exists.

\textbf{Independence.}
Let $e_0$ code a machine which enumerates formal proofs and halts upon
finding a contradiction in ZFC.
The construction can be formalized in ZFC, and thus
\[
  \tinv(I_{e_0})=4 \iff \text{ZFC is consistent.}
\]
If ZFC is consistent, G\"odel's second incompleteness theorem shows
that ZFC does not prove its own consistency, hence does not prove
$\tinv(I_{e_0})=4$.  The statement $\tinv(I_{e_0})=3$ asserts that some
exponent vectors and rational coefficients form a trinomial in
$I_{e_0}$.  Membership in $I_{e_0}$ is decidable, so this is an
existential arithmetic sentence, and it is false.  If ZFC proves only
true existential arithmetic sentences, then ZFC does not prove
$\tinv(I_{e_0})=3$ either.  Under the latter hypothesis, which implies
the former, the value of $\tinv(I_{e_0})$ is independent of ZFC.

\textbf{A fixed ideal.}
Leave the variable $E$ unspecialized in $U$ and apply
Theorem~\ref{thm:main} once. In the resulting Laurent ideal $J_U$, let $d_E$
denote the exponent coordinate corresponding to~$E$. Then
\[
  e \in K_0 \iff \text{$\tau_d \in J_U$ for some $d$ with $d_E=e$}.
\]
Thus one fixed ideal contains the halting problem in the fibers of its
trinomial exponents.

\section{Two effective cases} \label{sec:effective}

The ideals produced by our reduction are not radical, and this is
essential to our construction.  Their nilpotent structure is required for the quadratic
conditions on exponent vectors.  Theorem~\ref{thm:main} therefore
gives no conclusion for radical ideals.  This matches known
complexity results for ideal membership.
Bayer and Mumford explain in \cite{BayerMumford1993} that the doubly
exponential worst-case examples for ideal membership draw their
strength from embedded components of high multiplicity, while
membership in the radical admits singly exponential degree bounds
through the effective Nullstellensatz of Brownawell and Koll\'ar. We
close with two settings in which independent methods give terminating
algorithms.  The first is radical univariate ideals.

\begin{proposition} \label{prop:squarefree-univariate}
  Let $K$ be a number field and let $0\ne f\in K[x]$ be squarefree. It is
  decidable whether $(f)$ contains a nonzero polynomial with at most three
  terms.
\end{proposition}

\begin{proof}
  Multiplication by a power of $x$ does not change the number of
  terms, so we may remove the factor $x$ from $f$ and assume that
  $f(0)\ne0$. Let $\Omega\subseteq\ol K^\times$ be the set of roots
  of $f$. Partition $\Omega$ by the equivalence relation under
  which \(\alpha\sim\beta\) if and only if \(\alpha/\beta\) is a
  root of unity.  This partition and the orders of the roots of
  unity can be computed from the algebraic numbers in~$\Omega$.

  Suppose first that $\Omega$ has at most two equivalence classes. We
  can compute $r>0$ such that $\alpha^r$ is constant on each class.
  The resultant $h(y)=\operatorname{Res}_x\bigl(f(x),\,y-x^r\bigr)$
  lies in $K[y]$ and its roots are exactly the $\alpha^r$ with
  $\alpha\in\Omega$. Its squarefree part $g\in K[y]$ is computable
  and has degree $\bigl|\Set{\alpha^r:\alpha\in\Omega}\bigr|\le 2$
  because $\alpha^r$ takes one value on each equivalence class.
  Now $g(x^r)$ vanishes on $\Omega$, and since $f$ is squarefree,
  $g(x^r)$ is a multiple of $f$ with at most three terms.

  When $\Omega$ has at least three equivalence classes, Bilu and
  Luca give an explicit computable degree bound for every monic
  trinomial vanishing on~$\Omega$~\cite[Theorem~1.2]{BiluLuca2020}.
  Then simple enumeration of the finitely many exponent pairs
  within the bound plus linear algebra in $K[x]/(f)$ decides
  whether $(f)$ contains a polynomial with at most three terms.
\end{proof}

The second result is over finite fields.
\begin{proposition} \label{prop:finite-fields}
  Let $K$ be a finite field. There is an algorithm which, given an ideal
  $I\subseteq K[x_1,\ldots,x_n]$ and a positive integer $k$, decides whether
  $\tinv(I)\leq k$.
\end{proposition}

\begin{proof}
  As this is decidable, assume that $I$ contains no monomial.  Let
  $u = x_1 \cdots x_n$ and $L = K[x_1^{\pm1},\ldots,x_n^{\pm1}]$.
  Then $\tinv(I) = \tinv(I : u^\infty)$ since the inclusion
  $I \subseteq I : u^\infty$ gives $\geq$, and if $f \in I : u^\infty$
  then $u^q f \in I$ for some $q$, with support a translate of that
  of~$f$, which gives~$\leq$.  Every element of $IL$ is a monomial
  times an element of $I : u^\infty = IL \cap K[x_1, \ldots, x_n]$,
  so $\tinv(I)$ is the least number of terms of a nonzero element
  of~$IL$, and we may work in the Laurent quotient $A = L/IL \neq 0$.
  After translating one exponent and scaling the coefficient, an
  $r$-term relation in $A$ has the form
  \begin{equation}
    \label{eq:s-unit}
    1+c_1x^{a_1}+\cdots+c_{r-1}x^{a_{r-1}}=0
    \qquad (2\leq r\leq k),
  \end{equation}
  where the $c_i\in K^\times$ and the vectors
  $0,a_1,\ldots,a_{r-1}\in\ZZ^n$ are distinct. There are only finitely
  many coefficient tuples.  Write $K=\FF_{p^{e}}$ and let
  $\omega_1, \dots, \omega_e$ be a basis of $K$ over~$\FF_p$.  Then $L$
  is a free module with basis $\omega_1, \dots, \omega_e$ over
  $L_p = \FF_p[x_1^{\pm1},\ldots,x_n^{\pm1}]$, and $IL$ is the
  $L_p$-submodule generated by the products $\omega_j g_i$ of the basis
  elements with generators $g_1, \dots, g_s$ of~$I$.  Hence $A$ is a
  finitely presented $L_p$-module with a computable presentation, and
  \eqref{eq:s-unit} is the $S$-unit equation
  $x^{a_1} m_1 + \dots + x^{a_{r-1}} m_{r-1} = m_0$ over $A$ in the
  sense of~\cite{DongShafrir}, with $m_0 = -1$ and $m_i = c_i$.  For
  each coefficient tuple, its exponent solutions form an
  \emph{effective $p$-normal set} by~\cite[Theorem~1.3]{DongShafrir}.
  Such a set is effectively \emph{$p$-automatic}, meaning that a finite automaton can be
  computed that accepts it.  Effective automatic sets are closed under
  Boolean operations~\cite[Lemma~2.1]{DongShafrir}. We may therefore
  impose the conditions $a_i\ne0$ and $a_i\ne a_j$ for $i\ne j$,
  keeping a $p$-automatic exponent solution set, since the excluded
  sets are subgroups of $\ZZ^{(r-1)n}$ and subgroups are
  $p$-automatic~\cite[Section~2]{DongShafrir}.  Emptiness of the
  resulting automatic set is decidable.  Repeating this test for the
  finitely many values of $r$ and coefficient tuples proves the claim.
\end{proof}

\section*{Acknowledgements}

\setlength{\intextsep}{5pt}%
\setlength{\columnsep}{5pt}%
\begin{wrapfigure}{R}{0.12\linewidth}
\vspace{-.5\baselineskip}%
\centering%
\href{https://doi.org/10.3030/101110545}{%
\includegraphics[width=0.9\linewidth]{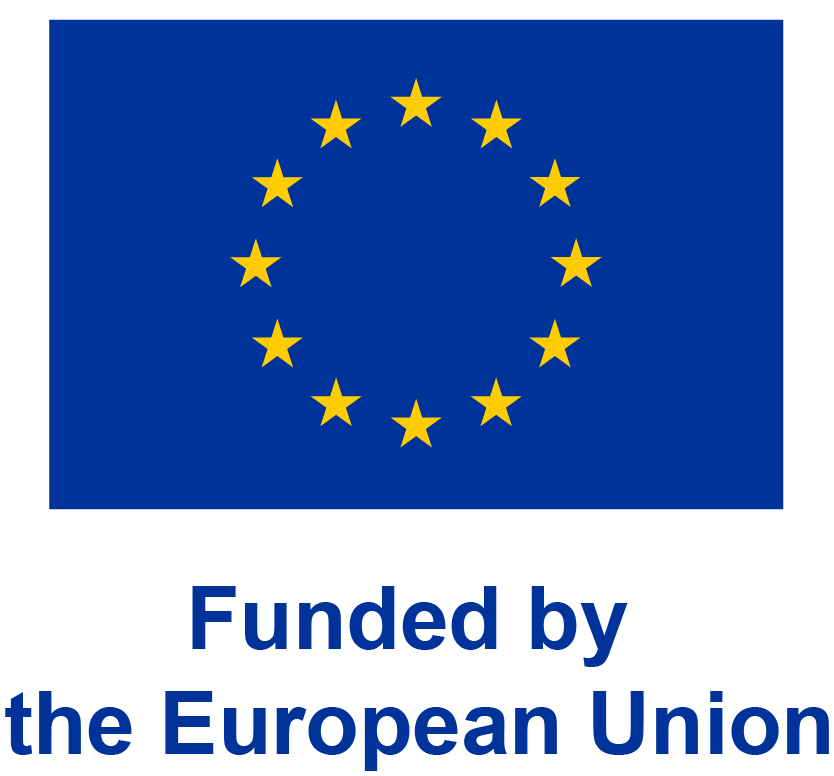}%
}
\end{wrapfigure}
T.B.~was funded by the European Union's Horizon 2020 research
and innovation programme under the Marie Skłodowska-Curie grant agreement
No.~101110545.
T.K.~and A.H.~are supported by the Deutsche Forschungsgemeinschaft
(DFG, German Research Foundation) within SPP~2458 ``Combinatorial
Synergies''\,-- 539866293.

\hphantom{a} \\
\hphantom{a}

\printbibliography

\end{document}